\documentclass[12pt]{amsart}
\usepackage{amsfonts}
\usepackage{amsfonts,latexsym,rawfonts,amsmath,amssymb,amsthm, a4,a4wide}

\usepackage[plainpages=false]{hyperref}

\usepackage{graphicx}

\RequirePackage{color}

\numberwithin{equation}{section}

\newcommand{\beq}{\begin{equation}}
\newcommand{\eeq}{\end{equation}}
\newcommand{\beqs}{\begin{eqnarray*}}
\newcommand{\eeqs}{\end{eqnarray*}}
\newcommand{\beqn}{\begin{eqnarray}}
\newcommand{\eeqn}{\end{eqnarray}}
\newcommand{\beqa}{\begin{array}}
\newcommand{\eeqa}{\end{array}}

\newcommand{\R}{\mathbb R}

\newcommand{\e}{\varepsilon}
\newcommand{\p}{\partial}

\newcommand{\diam}{\mbox{diam}\,}
\newcommand{\comment}[1]{}
\def\h{\hspace*{.24in}} 

{\begin{list}{}%
         {\setlength{\leftmargin}{#1}}%
         \item[]%
}
{\end{list}}

\newtheorem{prop}{Proposition}[section]
\newtheorem{thm}[prop]{Theorem}
\newtheorem{lem}[prop]{Lemma}

\newtheorem{rem}[prop]{Remark}

\newcommand{\dist}{\text{dist}}

\author{Nam Q. Le}
\address{Department of Mathematics, Indiana University, 
Bloomington, IN 47405, USA. }
\email {nqle@indiana.edu}

\thanks{The research of the  author was supported in part by NSF grant DMS-2054686.  }
\allowdisplaybreaks
\title[Global $W^{2,\delta}$ estimates for Monge-Amp\`ere]{On global $W^{2,\delta}$ estimates for the Monge-Amp\`ere equation on general bounded convex domains}

\makeatletter
\@namedef{subjclassname@2020}{%
  \textup{2020} Mathematics Subject Classification}
\makeatother

\begin{document}
\subjclass[2020]{35J25, 35J96, 35B45}
\keywords{Monge-Amp\`ere equation, global second derivative estimate, Pogorelov estimate}
\begin{abstract}
 We establish global $W^{2,\delta}$ estimates, for all $\delta<\frac{1}{n-1}$, for convex solutions to the Monge-Amp\`ere equation with positive $C^{2,\beta}$ right-hand side and zero boundary values on general bounded convex domains in $\R^n$ ($n\geq 2$).  We exhibit examples showing that global $W^{2, \frac{n}{2(n-1)}}$ estimates fail in all dimensions, so  the range of $\delta$ is sharp in two dimensions.
\end{abstract}

\maketitle

\section{Introduction and statement of the main result}
\label{Sect12}
This note is concerned with global second derivative estimates for the convex Aleksandrov solution to the Monge-Amp\`ere equation
\begin{equation}
 \label{MA1}
 \left\{
 \begin{alignedat}{2}
   \det D^{2} u~&= f \h~&&\text{in} ~\Omega, \\\
u &=0\h~&&\text{on}~\p \Omega
 \end{alignedat}
 \right.
\end{equation}
on general bounded convex domains $\Omega\subset\R^n$ ($n\geq 2$), where $f$ 
is bounded between two positive constants $\lambda\leq \Lambda$, that is,
\begin{equation}
\label{fbd}
0<\lambda\leq f\leq \Lambda.
\end{equation}
Regarding interior second-order Sobolev estimates, building on the work of  De Philippis--Figalli \cite{DPF},
 De Philippis--Figalli--Savin \cite{DPFS} and Schmidt \cite{Sc}, independently, show that $D^2 u\in L_{\mathrm{loc}}^{1+\e}(\Omega)$ for some constant $\e=\e(n,\lambda,\Lambda)>0$. If $f$ is assumed additionally to be continuous, then Caffarelli \cite{C2} shows that $u\in W^{2, p}_{\mathrm{loc}}(\Omega)$ for all
 $p\in (1,\infty)$. 
 
Regarding global second-order Sobolev estimates, when $\Omega$ is uniformly convex with $C^3$ boundary, 
Savin \cite{Sw2p} extends the above estimates all the way to the boundary by showing respectively that $D^2 u\in L^{1+\e}(\Omega)$, and $D^2 u\in L^{p}(\Omega)$ when $f\in C(\overline{\Omega})$.
The techniques in \cite{Sw2p} are based on the Boundary Localization Theorem established in \cite{SC2a}. 
In general, for the Monge-Amp\`ere equation with possibly nonzero boundary values, the uniform convexity of the boundary and the $C^3$ regularity of the boundary and boundary data are crucial for global $W^{2, p}$ estimates. In \cite{W}, Wang constructs explicit examples showing the failure
of global $W^{2, 3}$ estimates for the Monge--Amp\`ere
equation in two dimensions with positive constant right-hand side $f$ when either the boundary
data or the domain boundary failing to be $C^3$.

A natural question is to determine the optimal global integrability of the second derivatives for the solution $u$ to \eqref{MA1}--\eqref{fbd} when $\Omega$ is a general bounded convex domain. To the best of our knowledge, this issue has not been studied before. On the other hand, 
thanks to Caffarelli \cite{C1}, $|u|$ is known to grow at most like $[\dist(\cdot,\p\Omega)]^{2/n}$ away from the boundary. Therefore, 
by the convexity of $u$, $|Du|$  grows like $[\dist(\cdot,\p\Omega)]^{2/n-1}$ away from $\p\Omega$. These growths are shown to be optimal in the author's work \cite{L} for domains with  portions of $(n-1)$-dimensional hyperplanes on their boundaries. Given these optimal growths, it is reasonable to
expect that $\|D^2 u\|$  grows like $[\dist(\cdot,\p\Omega)]^{2/n-2}$ away from the boundary. This, in turn, indicates that the optimal global integrability for $D^2 u$ should be $L^\mu(\Omega)$ for all $\mu<\frac{n}{2(n-1)}$. We are able to confirm this expectation in two dimensions. For higher dimensions, there is still a gap
between our integrability result where $D^2 u\in L^\delta(\Omega)$ for all $\delta<\frac{1}{n-1}$, and the non-integrability examples for the threshold exponent $\frac{n}{2(n-1)}$. This is due to our method of proving the $W^{2,\delta}$ estimates; see Remark \ref{n23_rem} and Lemma \ref{Upmu}.

Our main result states as follows.

\begin{thm} 
\label{W2del_thm}
Let $u\in C(\overline{\Omega})$ be the  convex Aleksandrov solution to the Monge-Amp\`ere equation \eqref{MA1} where $\Omega$ is a bounded convex domain in  $\R^n$ ($n\geq 2$), and 
$f\in C^{2,\beta}(\overline\Omega)$ satisfies \eqref{fbd} where $\beta\in (0, 1)$. Then the following statements hold.
\begin{enumerate}
\item[(i)] For all $0<\delta<\frac{1}{n-1}$, we have $D^2 u\in L^\delta(\Omega)$ with estimate
\[\int_\Omega \|D^2 u\|^\delta\,dx \leq C(n,\Omega,\delta, \lambda,\Lambda, \|\log f\|_{C^2(\overline{\Omega})}). \]
\item[(ii)] If, in addition, $\Omega$ is a rectangular box, then  $D^2 u\not\in L^{\frac{n}{2(n-1)}}(\Omega)$.
\end{enumerate}
\end{thm}
In  the proof of Theorem \ref{W2del_thm}(i), we use Pogorelov-type estimates which require $u$ to be $C^4$. Therefore, it is natural to assume 
 $f\in C^{2,\beta}(\overline\Omega)$.
It would be interesting to reduce the regularity of $f$ in  Theorem \ref{W2del_thm}(i), and to improve the range of $\delta$ when $n\geq 3$.

The rest of this note is devoted to the proof of Theorem \ref{W2del_thm} and pertaining remarks.

\section{Proof of Theorem \ref{W2del_thm}}

Let $u$ be as in Theorem \ref{W2del_thm}.  Then $u$ is strictly convex; see Caffarelli \cite{C1} and also Figalli \cite[Corollary 4.11]{F}. Moreover, $u\in C^{4,\beta}(\Omega)$; see \cite[Theorem 3.10]{F}.

\subsection{Global $W^{2,\delta}$ estimates} 
We will establish the following pointwise Hessian estimates.
 \begin{lem} 
 \label{D2_ptw}
 Let $\Omega, u$, and $f$ be as in Theorem \ref{W2del_thm}(i). Let $\gamma\in (1, 2)$.
 Then, in $\Omega$, we have
\[\|D^2u(x)\|\leq 
\begin{cases}
C(n,\gamma,\Omega,\lambda,\Lambda, \|\log f\|_{C^2(\overline{\Omega})}) [\rm{dist}(x,\p\Omega)]^{-\gamma}&\text{when }n=2,\\
C(n,\Omega,\lambda,\Lambda, \|\log f\|_{C^2(\overline{\Omega})}) [\rm{dist}(x,\p\Omega)]^{1-n} &\text{when }n\geq 3.
\end{cases}
 \]
\end{lem}
\begin{rem}
Lemma \ref{D2_ptw} improves upon 
Theorem 3.9 in Figalli \cite{F} and Theorem 4.1 in Shi--Jiang \cite{SJ}, where 
the exponent in the Hessian estimate $\|D^2 u(x)\|\leq C[\dist(x,\p\Omega)]^{-\kappa}$ was, respectively, $-(3n+ 2)$ and $-(2n+\tau)$ where $\tau\in (1,2)$, instead of $\min\{-\gamma, 1-n\}$.
\end{rem}
Clearly, the global $W^{2,\delta}$ estimates in Theorem \ref{W2del_thm}(i)
are a consequence of
Lemma \ref{D2_ptw}.

It remains to prove Lemma \ref{D2_ptw}. One of our key tools is the following Pogorelov estimate, due to Trudinger and Wang \cite[Lemma 3.6]{TW}. 
\begin{lem} Let $v\in C^4(\overline{\Omega})$ be the convex solution to the Monge-Amp\`ere equation
\begin{equation*}
 \left\{
 \begin{alignedat}{2}
   \det D^{2} v~&= f \h~&&\text{in} ~\Omega, \\\
v &=0\h~&&\text{on}~\p \Omega,
 \end{alignedat}
 \right.
\end{equation*}
where $\Omega$ is a bounded convex domain in $\R^n$ ($n\geq 2$), and $f\in C^2(\overline{\Omega})$ with $f>0$ in $\overline{\Omega}$.
Then
\begin{equation}
\label{P_est}
|v(x)| \|D^2 v(x)\| \leq C(n, \|v\|_{L^{\infty}(\Omega)}, \|\log f\|_{C^2(\overline{\Omega})}) \big(1 + \|Dv\|^2_{L^{\infty}(\Omega}\big)\quad\text{in }\Omega.
\end{equation}
\end{lem}

\begin{proof}[Proof of Lemma \ref{D2_ptw}] We start with some general estimates for $u$. For the uniform estimate, we have (see \cite[Theorem 3.42]{LMT})
\[c(\Lambda, n) \|u\|^{n/2}_{L^{\infty}(\Omega)} \leq |\Omega|\leq C(\lambda, n) \|u\|^{n/2}_{L^{\infty}(\Omega)}, \]
where $c(\Lambda, n)>0$ and $C(\lambda, n)>0$, so
\begin{equation}
\label{umax_est}
 0<M_1(n, |\Omega|, \lambda)\leq \|u\|_{L^{\infty}(\Omega)}\leq M_2(n, |\Omega|, \Lambda).
 \end{equation}
 Since $u$ is convex and $u=0$ on $\p\Omega$, there holds
 \begin{equation}
 \label{u_dist}
 |u(x)|\geq \frac{\dist(x,\p\Omega)}{\diam (\Omega)} \|u\|_{L^{\infty}(\Omega)}\quad\text{for all } x\in\Omega,\end{equation}
 and
  \begin{equation}
 \label{Du_dist}
 |Du(x)| \leq \frac{|u(x)|}{\dist(x,\Omega)}\quad\text{for all } x\in\Omega.\end{equation}
We recall the following H\"older estimate, due to 
Caffarelli \cite[Lemma 1]{C1},
 \begin{equation}
 \label{C_est}
 |u(x)|\leq C_1(n,\alpha, \diam (\Omega),\Lambda)[\dist(x,\p\Omega)]^\alpha\quad\text{for all } x\in\Omega, \end{equation} 
 where
\begin{equation}
\label{aleq}
\alpha:=
\begin{cases}
\frac{2}{1+\gamma}\in (0, 1)& \text{when }n=2, \\ \frac{2}{n}&\text{when }n\geq 3.
 \end{cases}
 \end{equation}

For $h>0$ small, let
\[\Omega_h:= \{x\in \Omega: \dist(x,\p\Omega)>h\}\subset\subset\Omega,\]
and
\[A_h:=\{x\in \Omega: u(x)<-h\}\subset\subset\Omega.\]
From \eqref{u_dist}, we deduce that
\begin{equation}
\label{AOm}
A_h\supset \Omega_{\diam(\Omega) h/M_1}.
\end{equation}

Let $v:= u+ h$. Then, $v\in C^4(\overline A_h)$,  $v<0$ in $A_h$, and $v=0$ on $\p A_h$. Applying \eqref{P_est} to $v$ in $A_h$, and recalling \eqref{umax_est}, we find that
\begin{equation}
\label{Ph_est}
\sup_{A_h} \big(|u+ h| \|D^2 u\|\big) \leq C(n, |\Omega|, \Lambda, \|\log f\|_{C^2(\overline{\Omega})}) (1 + \|Du\|^2_{L^{\infty}(A_h)}).
\end{equation}

If $x\in A_h$, then $|u(x)|\geq h$, and \eqref{C_est} gives
\begin{equation}
\dist(x,\p\Omega) \geq c_1 h^{\frac{1}{\alpha}},\quad c_1=c_1(n,\alpha, \diam (\Omega),\Lambda)>0.
\end{equation}
Combining \eqref{Du_dist} and \eqref{C_est} with the above estimate, we obtain
\begin{equation}
\label{Du_est}
|Du(x)|  \leq C_1  [\dist(x,\p\Omega)]^{\alpha-1} \leq C_2 h^{1-\frac{1}{\alpha}}\quad \text{in } A_h.
\end{equation}
Thus, in $A_{2h}$ where $h$ is small, \eqref{Ph_est} and \eqref{Du_est} imply that
\begin{equation}
\label{D2A2h}
\|D^2 u\| \leq C (1 + \|Du\|^2_{L^{\infty}(A_h)})h^{-1}\leq C(n, |\Omega|, \Lambda, \|\log f\|_{C^2(\overline{\Omega})})  h^{1-\frac{2}{\alpha}}.
\end{equation}
It follows from \eqref{AOm} that
\begin{equation}
\label{D2uOh}
\|D^2 u\| \leq \bar C(n, |\Omega|, \Lambda, \|\log f\|_{C^2(\overline{\Omega})})  h^{1-\frac{2}{\alpha}}\quad\text{in } \Omega_{2\diam(\Omega) h/M_1}.
\end{equation}
In view of \eqref{aleq}, this easily concludes the proof of the lemma.
\end{proof}
\begin{rem} 
\label{n23_rem}
In the proof of Lemma \ref{D2_ptw}, we use both estimates \eqref{u_dist} and \eqref{C_est}. When $n=2$, by choosing $\gamma$ close to $1$, we see that the lower bound and the upper bound for $|u(x)|$ are almost of the same order in $\dist(x,\p\Omega)$. This is responsible for the sharp range of $\delta$ in Theorem \ref{W2del_thm}(i). However, for $n\geq 3$, 
the lower bound and the upper bound for $|u(x)|$ in \eqref{u_dist} and \eqref{C_est} are not of the same order. Thus, to obtain an improved range for $\delta$ when $n\geq 3$ without further assumptions on the geometry of $\Omega$, one needs completely different arguments. 
\end{rem}
We note that for $n\geq 3$, local improvements on the range of $\delta$ are possible when the boundary has flat portions. Due to Theorem \ref{W2del_thm} (ii), the exponent $\frac{n}{2(n-1)}$ in the next lemma is sharp.
\begin{lem}
\label{Upmu}
Let $u\in C(\overline{\Omega})$ be the  convex Aleksandrov solution to  \eqref{MA1} where $\Omega\supset (-2, 2)^{n-1}\times (0, 2)$ is a bounded convex domain in  $\R^n$ ($n\geq 3$) with $(-2, 2)^{n-1}\times\{0\}\subset\p\Omega$, and 
$f\in C^{2,\beta}(\overline\Omega)$ satisfies \eqref{fbd} where $\beta\in (0, 1)$. 
Then for $K:= (-1,1)^{n-1}\times (0, c)\subset\Omega$ where $c= c(n,\lambda, \Omega)\in (0, 1/4)$ is small, we have $D^2 u\in L^{\mu} (K)$ for all $\mu \in (0, \frac{n}{2(n-1)})$ with estimate
\[\|D^2 u\|_{L^{\mu}(K)} \leq C(n,\Omega, \lambda,\Lambda, \mu, \|\log f\|_{C^2(\overline{\Omega})}).\]
\end{lem}
\begin{proof} We use the same notation as in the proof of Lemma \ref{D2_ptw}. Our proof consists of improving \eqref{AOm} and \eqref{D2uOh}.

By \cite[Lemma 4.3]{L}, there exists $c_0= c_0(n,\lambda,\Omega)\in (0, 1/4)$ such that for
$K_0:= (-1, 1)^{n-1}\times (0, c_0)$,
we have
\[ |u(x)| \geq c_0 [\dist(x,\p\Omega)]^{\frac{2}{n}}\quad\text{if}\quad x\in K_0.\]
Therefore, for $0<h\leq c_0^2$, we obtain the following local improvement of \eqref{AOm}:
\begin{equation}
\label{AOm2}
A_h\cap K_0\supset \Omega_{c_0^{-n/2} h^{n/2}}\cap K_0.
\end{equation}
Using \eqref{D2A2h}, \eqref{aleq}, and \eqref{AOm2}, we find 
\begin{equation*}
\|D^2 u\| \leq \bar C(n, |\Omega|, \Lambda, \|\log f\|_{C^2(\overline{\Omega})})  h^{1-n}\quad\text{in } \Omega_{2 c_0^{-n/2} h^{n/2}}\cap K_0.
\end{equation*}
Consequently,
\begin{equation}
\label{D2uOh2}
\|D^2 u(x)\| \leq C(n, \Omega, \Lambda, \lambda, \|\log f\|_{C^2(\overline{\Omega})}) [\dist(x,\p\Omega)]^{\frac{2}{n}-2} \quad\text{in}\quad K,\end{equation}
for $K:= (-1,1)^{n-1}\times (0, c_1)\subset\Omega$ where $c_1= c_1(n,\lambda, \Omega)\in (0, 1/4)$ is small. This gives the conclusion of the lemma.
\end{proof}
\subsection{The rectangular box domain}\label{Sect22}
In this section, we prove Theorem \ref{W2del_thm}(ii) where $\Omega$ is a rectangular box. 
By the affine invariance of the Monge-Amp\`ere equation, we can assume, without loss of generality, that 
\[\Omega=(-1, 1)^{n-1}\times (0, 2).\]
Our main estimate, inspired by Wang \cite{W}, shows that  for a fixed positive fraction of \[x'\in Q_n:=[-1/2, 1/2]^{n-1},\]
$D_{nn} u (x', x_n) $ blows up like $ x_n^{\frac{2}{n}-2}$ when $x_n$ is small. This is the expected rate discussed in Section \ref{Sect12}. 

For $x\in\R^n$, we write $x=(x_1,\ldots, x_n)= (x', x_n)$ where $x'\in\R^{n-1}$.
Denote $D_i =\frac{\p}{\p x_i}$, and $D_{ij} =\frac{\p^2}{\p x_i \p x_j}$. Let $\mathcal{H}^s$ denote the $s$-dimensional Hausdorff measure. 
Below is our main measure-theoretic estimate.
\begin{lem}
\label{D2n_lem} Let $\Omega, u$, and $f$ be as in Theorem \ref{W2del_thm}(ii). Then,
for each $0<x_n<1/2$, there exists an $\mathcal{H}^{n-1}$ measurable subset $E_{x_n}\subset Q_n$ such that the following statements hold.
\begin{enumerate}
\item[(i)] $\mathcal{H}^{n-1} (E_{x_n})\geq 1/2$.
\item[(ii)] There exists a constant $c= c(n, \lambda,\Lambda)>0$ such that  for all $x'\in E_{x_n}$, we have
\begin{equation}
\label{Dnnu_ineq}
D_{nn} u (x', x_n) \geq 
\begin{cases}
c (x_n |\log x_n|)^{-1} & \text{when }n=2, \\
c x_n^{\frac{2}{n}-2} &\text{when }n\geq 3.
\end{cases}
\end{equation}
\end{enumerate}
\end{lem}
\begin{proof} 
We fix $x_n\in (0,1/2)$ in this proof. 

In view of the Hadamard determinant inequality (see \eqref{Ha_ineq}), to obtain \eqref{Dnnu_ineq}, it suffices to show that all the second pure derivatives $D_{ii} u(x', x_n)$ $(i=1,\ldots, n-1)$ are bounded from above by $Cx_n^{2/n}$ when $n\geq 3$, and by $Cx_n|\log x_n|$ when $n=2$. We will establish these bounds using one-dimensional slicing arguments.

When $n=2$, we can strengthen the H\"older estimate \eqref{C_est} to the following global log-Lipschitz estimate (see \cite[Proposition 1.4]{L})
\begin{equation}
 \label{Log_est}
 |u(x)|\leq C(\diam(\Omega), \Lambda) \dist (x,\p\Omega) (1+ |\log \dist(x,\p\Omega)|) \quad \text{for all }x\in\Omega\subset\R^2.\end{equation}
Now, if $x'\in Q_n$, then $\dist((x', x_n),\p\Omega)=x_n$, and thus
\eqref{C_est} and \eqref{Log_est} give
\begin{equation}
\label{n23_est}
|u(x', x_n)|\leq 
\begin{cases}
C_0(n,\Lambda)x_n|\log x_n|& \text{when }n=2, \\ C_0(n,\Lambda)x_n^{\frac{2}{n}}&\text{when }n\geq 3.
 \end{cases}
\end{equation}
Let 
\[\alpha:= \frac{2}{n},\quad a:= \frac{1}{n-1} (\frac{1}{2} + n-2).\]
Fix \[\tilde x=(x_2, \ldots, x_{n-1}, x_n)\quad \text{where }-\frac{1}{2}\leq x_i\leq \frac{1}{2} \quad \text{for }i=2,\ldots, n-2.\] 
We show that there exists a set $S_{\tilde x}\subset (-1,1)$ with $\mathcal{H}^1(S_{\tilde x})\geq a$ for which $D_{11} u(x_1, \tilde x)$, where $x_1\in S_{\tilde x}$, is bounded from above by $Cx_n^{2/n}$ when $n\geq 3$, and by $Cx_n|\log x_n|$ when $n=2$.

Indeed, by the convexity of $u$ and $u=0$ on $\p\Omega$, we have
\[0=u(1, \tilde x)\geq u(1/2, \tilde x) + D_1 u(1/2,\tilde x) (1/2).\]
Hence,
\[D_1 u(1/2, \tilde x)\leq -2u(1/2, \tilde x) = 2|u(1/2, \tilde x)|.\]
Similarly,
\[-D_1 u(-1/2, \tilde x)\leq 2|u(-1/2, \tilde x)|.\]
Therefore, invoking \eqref{n23_est}, we obtain a positive constant $C_1= 4 C_0(n,\Lambda)$ such that
\begin{equation}
\label{D1u_est}
D_1 u(1/2, \tilde x)-D_1 u(-1/2, \tilde x) 
\leq 
\begin{cases}
C_1(n,\Lambda)x_n|\log x_n|& \text{when }n=2, \\ C_1(n,\Lambda)x_n^\alpha&\text{when }n\geq 3.
 \end{cases}
\end{equation}

{\bf We first consider the case $n\geq 3$.} Let
\[S_{\tilde x}:=\Big\{x_1\in (-1/2, 1/2): D_{11} u(x_1, \tilde x) <\frac{C_1x_n^\alpha}{1-a}\Big \},\]
and \[ L_{\tilde x}: = (-1/2,1/2)\setminus S_{\tilde x}.\]
Then \[D_{11} u(x_1, \tilde x) \geq \frac{C_1x_n^\alpha}{1-a} \quad\text{for }x_1\in  L_{\tilde x}.\] Consequently,
\eqref{D1u_est} implies
\begin{eqnarray*}C_1x_n^{\alpha} \geq D_1 u(1/2, \tilde x)-D_1 u(-1/2, \tilde x) &=&  \int_{-1/2}^{1/2} D_{11} u(x_1, \tilde x) \,dx_1\\ & \geq& \int_{L_{\tilde x}} D_{11} u(x_1, \tilde x) \,dx_1\\
&\geq& \frac{C_1x_n^\alpha }{1-a} \mathcal{H}^{1} (L_{\tilde x}).
\end{eqnarray*}
It follows that 
\[\mathcal{H}^{1} (L_{\tilde x}) \leq 1-a,\]
and hence
\begin{equation}
\label{H1_est}
\mathcal{H}^{1} (S_{\tilde x})\geq a\quad\text{for each }
\tilde x=(x_2, \ldots, x_{n-1}, x_n)\quad \text{where } |x_i|\leq \frac{1}{2} (i=2,\ldots, n-2).\end{equation}
Let
\[E_{i, x_n}:=\Big\{x'\in Q_n: D_{ii} u(x', x_n) <\frac{C_1 x_n^\alpha}{1-a} \Big\},\]
and
\[E_{x_n}= \bigcap_{i=1}^{n-1}E_{i, x_n}.\]
Then, by \eqref{H1_est} and the Fubini Theorem, we have
\begin{equation}
\label{Ei_est}
\mathcal{H}^{n-1} (E_{i, x_n})\geq a.\end{equation}
Note that if $A$ and $B$ are two $\mathcal{H}^{n-1}$ measurable subsets of $Q_n$, then 
\[\mathcal{H}^{n-1}(A\cap B) = \mathcal{H}^{n-1} (A) + \mathcal{H}^{n-1} (B)-\mathcal{H}^{n-1}(A\cup B) \geq  \mathcal{H}^{n-1} (A) + \mathcal{H}^{n-1} (B)-1.\]
By induction, we then obtain from \eqref{Ei_est} that
\begin{equation}
\label{Exn_est}
 \mathcal{H}^{n-1} (E_{x_n})\geq \sum_{i=1}^{n-1} \mathcal{H}^{n-1} (E_{i, x_n}) -(n-2) \geq (n-1) a-(n-2)\geq 1/2. \end{equation}
For $x'\in E_{x_n}$, we have
\[D_{ii} u(x', x_n) \leq \frac{C_1x_n^\alpha}{1-a} \quad\text{for all } i=1,\ldots, n-1. \]
Thus, using the Hadamard determinant inequality
\begin{equation} 
\label{Ha_ineq}
\det D^2 u(x', x_n) \leq  \prod_{i=1}^{n} D_{ii} u(x', x_n), \end{equation}
together with $\det D^2 u(x', x_n)\geq\lambda$,  
we obtain
\begin{equation}
\label{Dnnu_est}
D_{nn} u(x', x_n) \geq \lambda (1-a)^{n-1}C_1^{1-n} x_n^{-(n-1)\alpha}=\lambda (1-a)^{n-1} C_1^{1-n} x_n^{\frac{2}{n}-2} \quad\text{for } x'\in E_{x_n}.\end{equation}
Due to \eqref{Exn_est} and \eqref{Dnnu_est}, the set $E_{x_n}$ satisfies the requirements of the lemma with $c=  \lambda (1-a)^{n-1} C_1^{1-n}$.

{\bf Finally, we consider the case $n=2$.} Then $a=1/2$. As above, it suffices to choose
\[E_{x_2}:=\{x_1\in (-1/2, 1/2): D_{11} u(x_1, x_2) <2C_1 x_2|\log x_2|\}.\]
The lemma is proved.
\end{proof}
\begin{proof}[Completion of the proof of Theorem \ref{W2del_thm}(ii).] We can assume $\Omega= (-1, 1)^{n-1}\times (0, 2)$.
Let $p>0$. Then,  Lemma \ref{D2n_lem} tells us that
\begin{eqnarray*}
\int_\Omega \|D^2 u\|^p\,dx &\geq &\int_0^{1/2} \int_{E_{x_n}} [D_{nn} u(x', x_n)]^p\, dx' dx_n\\
&\geq& 
\begin{cases}
\displaystyle
\frac{1}{2}\int_0^{1/2}(c (x_n |\log x_n|)^{-1})^p\,dx_n & \text{when }n=2, \\
\displaystyle
\frac{1}{2} \int_0^{1/2} (c x_n^{\frac{2}{n}-2})^p\,  dx_n &\text{when }n\geq 3
\end{cases}
=+\infty,
\end{eqnarray*}
if $p\geq \frac{n}{2(n-1)}$. This proves Theorem \ref{W2del_thm}(ii), and completes 
the proof of Theorem \ref{W2del_thm}.
\end{proof}
\section{Further remarks}
The method of the proof of Theorem \ref{W2del_thm}(ii) can be extended to singular and degenerate Monge-Amp\`ere equations. The following proposition is a representative.
\begin{prop}  Let $\Omega$ is a rectangular box  in  $\R^n$ ($n\geq 2$). Let $f\in C^{2,\beta}(\overline\Omega)$ be such that $0<\lambda\leq f\leq\Lambda$ where $\beta\in (0, 1)$.
Let $s\in (-\infty, n-2)$.
Let $u\in C(\overline{\Omega})$ be the  nonzero convex Aleksandrov solution to the Monge-Amp\`ere equation
\begin{equation}
 \label{MA2}
 \left\{
 \begin{alignedat}{2}
   \det D^{2} u~&= f|u|^{s} \h~&&\text{in} ~\Omega, \\\
u &=0\h~&&\text{on}~\p \Omega.
 \end{alignedat}
 \right.
\end{equation}
Then $D^2 u\not\in L^{\frac{n-s}{2(n-s)-2}}(\Omega)$ if $s\leq 0$, and $D^2 u\not\in L^{\frac{n-s}{2(n-s)-2}+\e}(\Omega)$ for any $\e>0$ if $s>0$.
\end{prop}
\begin{proof}
Following the proof of Proposition 2.8 in \cite{LSNS}, we have $u\in C^{4,\beta}(\Omega)$. The case $s=0$ follows from Theorem \ref{W2del_thm}(ii) so we only consider $s\neq 0$. We assume that $\Omega=(-1, 1)^{n-1}\times (0, 2)$, and use the same notation as in Section \ref{Sect22}. In particular, $x_n\in (0, 1/2)$. We consider two separate cases.

{\bf Case 1.} We first consider the case $s< 0$. In Lemma \ref{D2n_lem}, we replace \eqref{Dnnu_ineq} by 
\begin{equation}
\label{Dnnu_ineqs1}
D_{nn} u (x', x_n) \geq c x_n^{\frac{2-2(n-s)}{n-s}}
\end{equation}
where $c= c(n,\lambda,\Lambda, s)>0$, from which it follows that $D^2 u\not\in L^{\frac{n-s}{2(n-s)-2}}(\Omega)$.

To prove \eqref{Dnnu_ineqs1}, we make the following changes in the proof of Theorem \ref{W2del_thm}(ii). Due to \cite[Theorem 1.1 (i)]{LDCDS}, we can replace \eqref{n23_est} 
by
\begin{equation}
\label{n23_ests1}
|u(x', x_n)| \leq C_0(n,\Lambda, s) x_n^{\frac{2}{n-s}}.
\end{equation}
We replace $\alpha$ by 
\[\alpha_s:= \frac{2}{n-s}.\]
From \eqref{Ha_ineq} and 
\[\det D^2 u(x', x_n)\geq\lambda |u(x', x_n)|^s \geq \lambda (C_0 x_n^{\alpha_s})^s,\] we have, instead of \eqref{Dnnu_est},
\[D_{nn} u(x', x_n) \geq \lambda (C_0 x_n^{\alpha_s})^s C_1^{1-n} (1-a)^{n-1} x_n^{-(n-1)\alpha_s}= cx_n^{\frac{2-2(n-s)}{n-s}},\]
which is \eqref{Dnnu_ineqs1} where $c= \lambda C_0^s C_1^{1-n} (1-a)^{n-1}>0$.

{\bf Case 2.} We next consider the case $0<s<n-2$. Let \[0<\mu_1<\frac{2}{n-s}<\mu_2<1.\]
In Lemma \ref{D2n_lem}, we replace \eqref{Dnnu_ineq} by 
\begin{equation}
\label{Dnnu_ineqs2}
D_{nn} u (x', x_n) \geq c x_n^{s\mu_2-(n-1)\mu_1}
\end{equation}
where $c= c(n,\lambda,\Lambda, s,\mu_1,\mu_2)>0$. 

Thus, given any $\e>0$, we can choose $\mu_1$ and $\mu_2$ close to $\frac{2}{n-2}$ so that
\[(s\mu_2-(n-1)\mu_1)(\frac{n-s}{2(n-s)-2}+\e)\leq -1,\]
which shows that $D^2 u\not\in L^{\frac{n-s}{2(n-s)-2}+\e}(\Omega)$.

To prove \eqref{Dnnu_ineqs2}, we make the following changes in the proof of Theorem \ref{W2del_thm}(ii). Due to \cite[Proposition 1]{LDCDS}, 
we can replace \eqref{n23_est} 
by
\begin{equation}
\label{n23_ests2}
|u(x', x_n)| \leq C_0(n,\Lambda, s,\mu_1) x_n^{\mu_1}.
\end{equation}
We replace $\alpha$ by 
\[\alpha_s:= \mu_1.\]
By \cite[Theorem 1.1]{L}, we have
\[|u(x', x_n)|\geq c_1(n, s,\mu_2,\lambda) x_n^{\mu_2}.\]
From \eqref{Ha_ineq} and 
\[\det D^2 u(x', x_n)\geq\lambda |u(x', x_n)|^s \geq \lambda (c_1 x_n^{\mu_2})^s,\] we have, instead of \eqref{Dnnu_est},
\[D_{nn} u(x', x_n) \geq \lambda (c_1 x_n^{\mu_2})^s C_1^{1-n} (1-a)^{n-1} x_n^{-(n-1)\mu_1}= cx_n^{s\mu_2-(n-1)\mu_1},\]
which is \eqref{Dnnu_ineqs2} where $c= \lambda c_1^s C_1^{1-n} (1-a)^{n-1}>0$.

We have completed the proof of the proposition.
\end{proof}
\begin{rem} It would be interesting to establish an analogue of Theorem \ref{W2del_thm}(i) for \eqref{MA2} when $s\neq 0$. If we apply \eqref{P_est} as in the proof of Lemma \ref{D2_ptw}, then in \eqref{Ph_est}, the quantity
$ \|\log f\|_{C^2(\overline{\Omega})}$ has to be replaced by  $\|\log (f|u|^s)\|_{C^2(\overline{\Omega})}$ which we do not have a priori control.
\end{rem}

\end{document}